\documentclass[11pt]{amsart}

\usepackage{amssymb}
\usepackage{enumerate}
\usepackage{graphicx}
\usepackage[margin=1in]{geometry}
\usepackage[colorlinks=true, citecolor=blue, linkcolor=blue, urlcolor=black]{hyperref}

\usepackage{array}

\usepackage{url} 
\usepackage{mathrsfs} 
\usepackage{wasysym} 

\DeclareMathAlphabet{\mathpzc}{OT1}{pzc}{m}{it} 

\newtheorem{theorem}{Theorem}[section]
\newtheorem{corollary}[theorem]{Corollary}
\newtheorem{lemma}[theorem]{Lemma}
\newtheorem{proposition}[theorem]{Proposition}

\newtheorem{conjecture}[theorem]{Conjecture}
\newtheorem{defi}[theorem]{Definition}
\newtheorem{problem}[theorem]{Problem}

\numberwithin{equation}{section}

\newcommand{\PP}{\mathbb{P}}
\newcommand{\N}{\mathbb{N}}
\newcommand{\Z}{\mathbb{Z}}

\newcommand{\R}{\mathbb{R}}
\newcommand{\C}{\mathbb{C}}
\newcommand{\pun}[1]{\lceil #1 \rceil}
\newcommand{\pin}[1]{\lfloor #1 \rfloor}

\DeclareMathOperator{\Dom}{Dom}
\DeclareMathOperator{\Res}{Res}
\DeclareMathOperator{\li}{li}

\newcommand{\disum}{\displaystyle \sum}
\newcommand{\diint}{\displaystyle \int}

\newcommand{\dilim}{\displaystyle \lim}
\newcommand{\diprod}{\displaystyle \prod}
\newcommand{\difrac}{\displaystyle \frac}
\newcommand{\diliminf}{\displaystyle \liminf}
\newcommand{\dilimsup}{\displaystyle \limsup}

\newcommand{\dibinom}{\displaystyle \binom}

\begin{document}

\title{Real exponential sums over primes and prime gaps}

\author{Luan Alberto Ferreira}

\address{IFSP, \textit{campus} Itaquaquecetuba. Rua Primeiro de Maio, 500, Itaquaquecetuba - SP, Brazil. 08571-050.}

\address{IME - USP. Rua do Mat\~ao, 1010, S\~ao paulo - SP, Brazil. 05508-090.}

\email{luan.ferreira@ifsp.edu.br}

\begin{abstract}

We prove that given $\lambda \in \R$ such that $0 < \lambda < 1$, then $\pi(x + x^\lambda) - \pi(x) \sim \displaystyle \frac{x^\lambda}{\log(x)}$. This solves a long-standing problem concerning the existence of primes in short intervals. In particular, we give a positive answer (for all sufficiently large number) to some old conjectures about prime numbers, such as Legendre's conjecture about the existence of at least two primes between two consecutive squares.

\end{abstract}

\date{\today}
\subjclass[2020]{Primary 11N05, Secondary 11L20}

\maketitle





\section{Introduction}

\subsection{Purpose of the paper}

Let $\mathbb{P}$ the set of the prime numbers, $\pi$ the prime counting function, $\lambda \in \R$ such that $0 < \lambda < 1$, and $\exp$ the exponential function $x \mapsto \exp(x) = \disum_{n = 0}^{\infty} \difrac{x^n}{n!}$. This article proves that given $0 < c < \difrac{1}{1 - \lambda}$, then \begin{equation} \label{eq1.1} \sum_{p \leq x} \difrac{c(1 - \lambda)\log(p)\exp(cp^{1 - \lambda})}{p^{\lambda}} \sim \exp(cx^{1 - \lambda}),\end{equation} as suggested in \cite{RS}. This implies, in particular, that \begin{equation} \label{eq1.2} \pi(x + x^\lambda) - \pi(x) \sim \frac{x^\lambda}{\log(x)}. \end{equation}

\subsection{Organization of the paper}

The main section of this article is section \ref{section2}, in which proofs of formulas \ref{eq1.1} and \ref{eq1.2} are presented. This is done through an adaptation of Newman's ingenious proof \cite{Newman} of the Prime Number Theorem (PNT). To do this we will follow closely the presentation given in \cite{Sutherland}, but the reader could also see \cite{BN}, \cite{Jameson}, \cite{Keng}, \cite{KL}, \cite{Lang}, \cite{ORourke} or \cite{Zagier}. Section \ref{section3} contains some technical lemmas that are used throughout the text, while in section \ref{section4} some considerations are made about the result obtained, in particular, we briefly explain how the main result implies some old conjectures about prime numbers, and the author takes the liberty to tell a bit about where the solution of this problem came from. \vspace{10mm}

\makeatletter
\enddoc@text
\let\enddoc@text\empty 
\makeatother

\newpage

\section{The adaptation of Newman's proof} \label{section2}

The PNT, i.e., the affirmation $$\pi(x) \sim \difrac{x}{\log(x)},$$

\noindent is among the most beautiful and celebrated theorems in all of mathematics. In 1980, D. J. Newman found an ingenious proof of this statement, which can be the summarized in the following $10$ steps:

\noindent \underline{Step 1:} Let $\theta(x) = \disum_{p \leq x} \log(p)$ the first Chebyshev function. Show that $\theta(x) \sim x$ implies PNT.\footnote{In fact, these two statements are equivalents, but we only need one direction.}

\noindent \underline{Step 2:} Show that if the integral $\diint_{1}^{\infty} \difrac{\theta(t) - t}{t^2} dt$ converges, then $\theta(x) \sim x$.

\noindent \underline{Step 3:} Make the substitution $t = \exp(x)$ to write the integral from step $2$ as follows: $$\diint_{1}^{\infty} \difrac{\theta(t) - t}{t^2} dt = \diint_{0}^{\infty} \difrac{\theta(\exp(x)) - \exp(x)}{\exp(x)\exp(x)} \exp(x) dx = \diint_{0}^{\infty} \difrac{\theta(\exp(x))}{\exp(x)} - 1 dx.$$

\noindent \underline{Step 4:} Show that $\theta(x) \leq \log(4)x$, $\forall \ x \geq 1$.

\noindent \underline{Step 5:} (Newman's analytic theorem) Let $h: [0, \infty) \to \R$ be a bounded and locally integrable function whose Laplace transform $$\mathscr{L}(s) = \diint_{0}^{\infty} h(x)\exp(-sx) dx,$$

\noindent initially defined for $\Re(s) > 0$, extends analytically to $\Re(s) \geq 0$. Then the improper integral $$\diint_{0}^{\infty} h(x) dx$$

\noindent converges and its value is $\mathscr{L}(0)$.

\noindent \underline{Step 6:} Let $\Phi(s) = \disum_p \difrac{\log(p)}{p^s}$, initially defined for all $s \in \C$ such that $\Re(s) > 1$. Show that the Laplace transform of the integrand of the step $3$ is $$\mathscr{L}(s) = \difrac{1}{s+1} \sum_p \difrac{\log(p)}{p^{s + 1}} - \difrac{1}{s} = \difrac{\Phi(s + 1)}{s + 1} - \difrac{1}{s},$$

\noindent defined for all $s \in \C$ such that $\Re(s) > 0$.

\noindent \underline{Step 7:} Show that the function in the previous step can be extend analytically to $\Re(s) \geq 0$ if, and only if, the function $$s \mapsto \Phi(s) - \difrac{1}{s - 1}$$

\noindent can be extend analytically to $\Re(s) \geq 1$.

\noindent \underline{Step 8:} Use the identity $$\difrac{1}{z - 1} = \difrac{1}{z} + \difrac{1}{z(z - 1)},$$

\noindent valid to all $z \in \C - \{0, 1\}$ to show that the identity $$\sum_p \difrac{\log(p)}{p^s - 1} = \Phi(s) + \sum_p \difrac{\log(p)}{p^s(p^s - 1)}$$

\noindent is valid for all $s \in \C$ such that $\Re(s) > 1$.

\noindent \underline{Step 9:} Show that the function $$s \mapsto \sum_p \difrac{\log(p)}{p^s(p^s - 1)}$$

\noindent is analytic on $\Re(s) > \difrac{1}{2}$. This reduces the problem to show that the function $$s \mapsto \sum_p \difrac{\log(p)}{p^s - 1} - \difrac{1}{s - 1}$$ can be extended analytically to $\Re(s) \geq 1$.

\noindent \underline{Step 10:} Identify the function $$s \mapsto -\sum_p \difrac{\log(p)}{p^s - 1}$$

\noindent as the logarithmic derivative of a function $\zeta$ (the Riemann zeta function) which has the following properties:

\begin{itemize}
    \item $\zeta(s) = \diprod_p \difrac{1}{1 - p^{-s}} = \sum_{n = 1}^{\infty} \difrac{1}{n^s}$, for all $s \in \C$ such that $\Re(s) > 1$;\footnote{And it is at this point in the argument that we use the Fundamental Theorem of Arithmetic (FTA).}
    \item $\zeta(s) - \difrac{1}{s - 1}$ can be extended analytically to $\Re(s) \geq 1$;\footnote{In fact, the function $\zeta(s) - \difrac{1}{s - 1}$ can be extended analytically to the whole complex plane, but we only need that the analytic extension exists in an open set containing $\Re(s) \geq 1$.}
    \item $\zeta(s) \neq 0$, for all $s \in \C$ such that $\Re(s) \geq 1$;
\end{itemize}

This finishes Newman's proof of PNT. In what follows, we will show that this script can be somehow adapted to prove formula \ref{eq1.2}. This relies on a weight function properly chosen to work in a similar way to $\log(x)$ on Chebyshev's function $\theta(x)$. So, let $0 < \lambda < 1$.

\begin{defi}
For each $c > 0$, define for $x > 0$ $$w(x) = \difrac{c(1 - \lambda)\log(x)\exp(cx^{1 - \lambda})}{x^{\lambda}} \qquad \qquad \text{and} \qquad \qquad W(x) = \disum_{p \leq x} w(p)$$ 
\end{defi}

We will leave the explanation of why the function $w$ was chosen in this way for section \ref{section4}. For now, just note that $w$ is increasing for $x$ sufficiently large and that $$\diint_{2}^{x} \difrac{w(t)}{\log(t)} dt = \diint_{2}^{x} \difrac{c(1 - \lambda)\exp(ct^{1 - \lambda})}{t^{\lambda}} dt = \exp(cx^{1 - \lambda}) - \exp(c2^{1 - \lambda}) \sim \exp(cx^{1 - \lambda}),$$

\noindent so that our goal will be to show that $$W(x) \sim \exp(cx^{1 - \lambda}).$$

\begin{center} \underline{Step 1} \end{center}

Step $1$ can be performed in a much more general way. To do so, fix the following notation: if $A \subseteq \N = \{n \in \Z; \ n > 0\}$ and $x \in \R$, then $A[x]$ will denote $\{a \in A; \ a \leq x\}$ and $\pi_A(x)$ will denote the cardinality of $A[x]$.

\begin{proposition} \label{initialproposition}

Let $A \subseteq \N$ and suppose that for some $c > 0$ we have $$\disum_{a \in A[x]} w(a) \sim \exp(cx^{1 - \lambda}).$$ Then $$\dilim_{x \to \infty} \difrac{\disum_{a \in A[x+x^\lambda]} w(a)}{\disum_{a \in A[x]} w(a)} = \exp[c(1 - \lambda)].$$

\end{proposition}

\begin{proof} By hypothesis, $$\difrac{\disum_{a \in A[x+x^\lambda]} w(a)}{\disum_{a \in A[x]} w(a)} \sim \difrac{\exp[c(x + x^{\lambda})^{1 - \lambda}]}{\exp(cx^{1 - \lambda})} \xrightarrow[x \to \infty]{} \exp[c(1 - \lambda)],$$

\noindent by lemma \ref{lemma1}. \end{proof}

The next two propositions will make it clear why we entered the parameter $c$ in the function $w$.

\begin{proposition} \label{limsupproposition}

Let $A \subseteq \N$ and suppose that for some $c > 0$ we have $$\disum_{a \in A[x]} w(a) \sim \exp(cx^{1 - \lambda}).$$ Then $$\dilimsup_{x \to \infty} \difrac{\pi_A(x + x^\lambda) - \pi_A(x)}{x^\lambda/\log(x)} \leq \difrac{\exp[c(1 - \lambda)] - 1}{c(1 - \lambda)}.$$

\end{proposition}

\begin{proof} Suppose by absurd that for some $\varepsilon > 0$ there are arbitrarily large $x \in \R$ such that $$\pi_A(x + x^\lambda) - \pi_A(x) \geq \left(\difrac{\exp[c(1 - \lambda)] - 1}{c(1 - \lambda)} + \varepsilon\right) \difrac{x^\lambda}{\log(x)}$$ and let $$u = \left(\difrac{\exp[c(1 - \lambda)] - 1}{c(1 - \lambda)} + \varepsilon\right) \difrac{x^\lambda}{\log(x)}$$ to simplify the notation.

For these $x \in \R$, and working where $w$ is increasing, \begin{equation*} \begin{split} \disum_{\substack{a \in A \\ x < a \leq x+x^\lambda}} w(a) & \geq \disum_{j = 0}^{\pun{u} - 1} w(x+j) \\ & \geq \diint_{x-1}^{x + \pun{u} - 1} w(t) dt \\ & \geq \diint_{x}^{x + u - 1} w(t) dt \\ & = \diint_{x}^{x + u} w(t) dt - \diint_{x + u - 1}^{x + u} w(t) dt \\ & \geq \diint_{x}^{x + u} w(t) dt - w(x + u). \end{split} \end{equation*}

This implies \begin{equation*} \begin{split} \difrac{\disum_{a \in A[x+x^\lambda]} w(a)}{\disum_{a \in A[x]} w(a)} & = 1 + \difrac{\disum_{\substack{a \in A \\ x < a \leq x+x^\lambda}} w(a)}{\disum_{a \in A[x]} w(a)} \\ & \geq 1 + \difrac{\diint_{x}^{x + u} w(t) dt}{\disum_{a \in A[x]} w(a)} - \difrac{w(x + u)}{\disum_{a \in A[x]} w(a)} \\ & \sim 1 + \difrac{\diint_{x}^{x + u} w(t) dt}{\exp(cx^{1 - \lambda})} - \difrac{w(x + u)}{\exp(cx^{1 - \lambda})} \\ & \xrightarrow[x \to \infty]{} 1 + \left(\difrac{\exp[c(1 - \lambda)] - 1}{c(1 - \lambda)} + \varepsilon\right)c(1 - \lambda) - 0 \\ & = \exp[c(1 - \lambda)] + c\varepsilon(1 - \lambda),\end{split} \end{equation*} by lemma \ref{integrallimsup}. Hence, by proposition \ref{initialproposition},

\vspace{-3mm}

$$\exp[c(1 - \lambda)] \geq \exp[c(1 - \lambda)] + c\varepsilon(1 - \lambda),$$

\vspace{1mm}

\noindent which implies, in turn, that $\varepsilon \leq 0$, absurd. \end{proof}

\begin{proposition} \label{liminfproposition}

Let $A \subseteq \N$ and suppose that for some $c > 0$ we have $$\disum_{a \in A[x]} w(a) \sim \exp(cx^{1 - \lambda}).$$ Then $$\diliminf_{x \to \infty} \difrac{\pi_A(x + x^\lambda) - \pi_A(x)}{x^\lambda/\log(x)} \geq \difrac{1 - \exp[c(\lambda - 1)]}{c(1 - \lambda)}.$$

\end{proposition}

\begin{proof} Suppose by absurd that for some $\varepsilon > 0$ there are arbitrarily large $x \in \R$ such that $$\pi_A(x + x^\lambda) - \pi_A(x) \leq \left(\difrac{1 - \exp[c(\lambda - 1)]}{c(1 - \lambda)} - \varepsilon\right) \difrac{x^\lambda}{\log(x)}$$ and let $$v = \left(\difrac{1 - \exp[c(\lambda - 1)]}{c(1 - \lambda)} - \varepsilon\right) \difrac{x^\lambda}{\log(x)}$$ to simplify the notation.

For these $x \in \R$, and working where $w$ is increasing, \begin{equation*} \begin{split} \disum_{\substack{a \in A \\ x < a \leq x+x^\lambda}} w(a) & \leq \disum_{j = 0}^{\pin{v} - 1} w(x + x^\lambda - j) \\ & \leq \diint_{x + x^\lambda - \pin{v} + 1}^{x + x^\lambda + 1} w(t) dt \\ & \leq \diint_{x + x^\lambda - v}^{x + x^\lambda + 1} w(t) dt \\ & = \diint_{x + x^\lambda - v}^{x + x^\lambda} w(t) dt + \diint_{x + x^\lambda}^{x + x^\lambda + 1} w(t) dt \\ & \leq \diint_{x + x^\lambda - v}^{x + x^\lambda} w(t) dt + w(x + x^\lambda + 1). \end{split} \end{equation*}

This implies \begin{equation*} \begin{split} \difrac{\disum_{a \in A[x+x^\lambda]} w(a)}{\disum_{a \in A[x]} w(a)} & = 1 + \difrac{\disum_{\substack{a \in A \\ x < a \leq x+x^\lambda}} w(a)}{\disum_{a \in A[x]} w(a)} \\ & \leq 1 + \difrac{\diint_{x + x^\lambda - v}^{x + x^\lambda} w(t) dt}{\disum_{a \in A[x]} w(a)} + \difrac{w(x + x^\lambda + 1)}{\disum_{a \in A[x]} w(a)} \\ & \sim 1 + \difrac{\diint_{x + x^\lambda - v}^{x + x^\lambda} w(t) dt}{\exp(cx^{1 - \lambda})} + \difrac{w(x + x^\lambda + 1)}{\exp(cx^{1 - \lambda})} \\ & \xrightarrow[x \to \infty]{} 1 + \left(\difrac{1 - \exp[c(\lambda - 1)]}{c(1 - \lambda)} - \varepsilon\right)c(1 - \lambda)\exp[c(1 - \lambda)] + 0 \\ & = \exp[c(1 - \lambda)][1 - c\varepsilon(1 - \lambda)], \end{split} \end{equation*} by lemma \ref{integralliminf}. Hence, by proposition \ref{initialproposition},

\vspace{-3mm}

$$\exp[c(1 - \lambda)] \leq \exp[c(1 - \lambda)][1 - c\varepsilon(1 - \lambda)],$$

\vspace{1mm}

\noindent which implies, in turn, that $\varepsilon \leq 0$, absurd. \end{proof}

\begin{corollary} \label{abstractmainresult}

Let $A \subseteq \N$. If exists a sequence $\{c_n\}_{n \in \N}$ of positive real numbers such that $$\dilim_{n \to \infty} c_n = 0$$ and $$\disum_{a \in A[x]} \difrac{c_n(1 - \lambda)\log(a)\exp(c_n a^{1 - \lambda})}{a^{\lambda}} \sim \exp(c_n x^{1 - \lambda}), \qquad \forall \ n \in \N,$$ then $$\pi_A(x + x^\lambda) - \pi_A(x) \sim \difrac{x^\lambda}{\log(x)}.$$

\end{corollary}

\begin{proof} By propositions \ref{limsupproposition} and \ref{liminfproposition}, for each $n \in \N$ we have

$$\difrac{1-\exp[c_n(\lambda - 1)]}{c_n(1 - \lambda)} \leq \diliminf_{x \to \infty} \difrac{\pi_A(x + x^\lambda) - \pi_A(x)}{x^\lambda/\log(x)}$$

\noindent and

$$\dilimsup_{x \to \infty} \difrac{\pi_A(x + x^\lambda) - \pi_A(x)}{x^\lambda/\log(x)} \leq \difrac{\exp[c_n(1 - \lambda)] - 1}{c_n(1 - \lambda)}.$$

\noindent Taking $n \to \infty$, we obtain the result. \end{proof}

A good indication that the weight function $W$ is adequate to address this problem is that, under the Riemann hypothesis, we can already prove that formula \ref{eq1.2} is valid for all $\difrac{1}{2} < \lambda < 1$. This relies on the following theorem, which comes from \cite{RS}.

\begin{theorem} [Rosser-Schoenfeld] \label{RS}

Let $f$ be a continuously differentiable real function defined on an open interval containing $[2, \infty)$, and let $\pi(x) = \li(x) + \epsilon(x)$, where $\li(x) = \diint_{2}^{x} \difrac{1}{\log(t)} dt$. Then $$\disum_{p \leq x} f(p) = \diint_{2}^{x} \difrac{f(t)}{\log(t)} dt + \epsilon(x)f(x) - \diint_{2}^{x} \epsilon(t)f'(t) dt, \ \forall \ x \geq 2.$$

\end{theorem}

\begin{proof} See \cite{BS}. \end{proof}

\begin{corollary}

Under the Riemann hypothesis, formula \ref{eq1.2} is valid for all $\difrac{1}{2} < \lambda < 1$.

\end{corollary}

\begin{proof} Let $\difrac{1}{2} < \lambda < 1$ and $c > 0$. By theorem \ref{RS}, for $x \geq 2$, $$\disum_{p \leq x} w(p) = \diint_{2}^{x} \difrac{w(t)}{\log(t)} dt + \epsilon(x)w(x) - \diint_{2}^{x} \epsilon(t)w'(t) dt.$$

Under the Riemann Hypothesis\footnote{In fact, it is equivalent.}, there exists $R > 0$ such that $$|\epsilon(x)| \leq Rx^{1/2}\log(x), \qquad \forall \ x \geq 2.$$

Thus $$\left| \difrac{\epsilon(x)w(x)}{\exp(cx^{1 - \lambda})} \right| \leq \difrac{Rx^{1/2}\log(x) \cdot c(1 - \lambda)\log(x)}{x^{\lambda}} \xrightarrow[x \to \infty]{} 0$$ and \begin{equation*} \begin{split} \difrac{\left|\diint_{2}^{x} \epsilon(t)w'(t) dt \right|}{\exp(cx^{1 - \lambda})} & \leq \difrac{\diint_{2}^{x} Rt^{1/2}\log(t)|w'(t)| dt}{\exp(cx^{1 - \lambda})} \\ & \overset{\text{LH}}{=} \difrac{Rx^{1/2}\log(x)w'(x)}{\exp(cx^{1 - \lambda})c(1 - \lambda)x^{-\lambda}} \\ & = \difrac{Rx^{1/2}\log(x)[1 + c(1 - \lambda)\log(x)x^{1 - \lambda} - \lambda \log(x)]}{x} \\ & \xrightarrow[x \to \infty]{} 0, \end{split} \end{equation*} which implies $$\disum_{p \leq x} w(p) \sim \exp(cx^{1 - \lambda}).$$

Now the result follows from corollary \ref{abstractmainresult}. \end{proof}

But we will show the validity of formula \ref{eq1.2} (for all $0 < \lambda < 1$) regardless of the Riemann hypothesis. Let's do it!

\begin{center} \underline{Step 2} \end{center}

\begin{proposition} \label{step2}

Let $0 < c < \difrac{1}{1 - \lambda}$. If $$\diint_{1}^{\infty} \difrac{W(t) - \exp(ct^{1 - \lambda})}{t^{\lambda}\exp(ct^{1 - \lambda})} dt$$ converges, then $W(x) \sim \exp(cx^{1 - \lambda})$.

\end{proposition}

\begin{proof} Suppose that $W(x) \not \sim \exp(cx^{1 - \lambda})$. We have two cases to analyze.

\noindent \underline{Case 1:} There exists $\varepsilon > 0$ such that $W(x) \geq (1 + \varepsilon)\exp(cx^{1 - \lambda})$ for arbitrarily large $x$. We can assume WLOG that $0 < \varepsilon < 1$. Since $$\difrac{1 - \exp[cx^{1 - \lambda} - c(x + \varepsilon x^\lambda)^{1 - \lambda}]}{c(1 - \lambda)} \xrightarrow[x \to \infty]{} \difrac{1 - \exp[-c\varepsilon (1 - \lambda)]}{c(1 - \lambda)}$$ from below (by lemma \ref{lemmah1}), given $$\delta = \difrac{\varepsilon^2}{56(1 + \varepsilon)} > 0$$ we have $$\difrac{1 - \exp[cx^{1 - \lambda} - c(x + \varepsilon x^\lambda)^{1 - \lambda}]}{c(1 - \lambda)} > \difrac{1 - \exp[-c\varepsilon (1 - \lambda)]}{c(1 - \lambda)} - \delta,$$ for all sufficiently large $x$. Therefore,

\begin{equation*} \begin{split} \int_{x}^{x + \varepsilon x^\lambda} \difrac{W(t) - \exp(ct^{1 - \lambda})}{t^\lambda \exp(ct^{1 - \lambda})} dt & \geq \int_{x}^{x + \varepsilon x^\lambda} \difrac{W(x) - \exp(ct^{1 - \lambda})}{t^\lambda \exp(ct^{1 - \lambda})} dt \\ & \geq \int_{x}^{x + \varepsilon x^\lambda} \difrac{(1 + \varepsilon)\exp(cx^{1 - \lambda}) - \exp(ct^{1 - \lambda})}{t^\lambda \exp(ct^{1 - \lambda})} dt \\ & = (1 + \varepsilon) \exp(cx^{1 - \lambda}) \int_{x}^{x + \varepsilon x^\lambda} \difrac{1}{t^\lambda \exp(ct^{1 - \lambda})} dt - \int_{x}^{x + \varepsilon x^\lambda} \difrac{1}{t^\lambda} dt \\ & = (1 + \varepsilon) \difrac{1 - \exp[cx^{1 - \lambda} - c(x + \varepsilon x^{\lambda})^{1 - \lambda}]}{c(1 - \lambda)} - \difrac{(x + \varepsilon x^{\lambda})^{1 - \lambda} - x^{1 - \lambda}}{1 - \lambda}\\ & > (1 + \varepsilon) \left[\difrac{1 - \exp[-c\varepsilon (1 - \lambda)]}{c(1 - \lambda)} - \delta\right] - \varepsilon \\ & = (1 + \varepsilon) \difrac{1 - \exp[-c\varepsilon (1 - \lambda)]}{c(1 - \lambda)} - \varepsilon - (1 + \varepsilon) \delta \\ & > \difrac{\varepsilon^2}{7} - (1 + \varepsilon) \difrac{\varepsilon^2}{56(1 + \varepsilon)} = \difrac{\varepsilon^2}{8} > 0, \end{split} \end{equation*} for arbitrarily large $x$, by lemma \ref{lemmat1}. This contradicts the fact that the integral converges.

\noindent \underline{Case 2:} There exists $\varepsilon > 0$ such that $W(x) \leq (1 - \varepsilon)\exp(cx^{1 - \lambda})$ for arbitrarily large $x$. We can assume WLOG that $0 < \varepsilon < 1$. Since $$\difrac{\exp[cx^{1 - \lambda} - c(x - \varepsilon x^\lambda)^{1 - \lambda}] - 1}{c(1 - \lambda)} \xrightarrow[x \to \infty]{} \difrac{\exp[c \varepsilon (1 - \lambda)] - 1}{c(1 - \lambda)}$$ from above (by lemma \ref{lemmah2}), given $$\delta = \difrac{\varepsilon^2}{6(1 - \varepsilon)} > 0$$ we have $$\difrac{\exp[cx^{1 - \lambda} - c(x - \varepsilon x^\lambda)^{1 - \lambda}] - 1}{c(1 - \lambda)} < \difrac{\exp[c \varepsilon (1 - \lambda)] - 1}{c(1 - \lambda)} + \delta,$$ for all sufficiently large $x$. Therefore, \begin{equation*} \begin{split} \int_{x - \varepsilon x^\lambda}^{x} \difrac{W(t) - \exp(ct^{1 - \lambda})}{t^\lambda \exp(ct^{1 - \lambda})} dt & \leq \int_{x - \varepsilon x^\lambda}^{x} \difrac{W(x) - \exp(ct^{1 - \lambda})}{t^\lambda \exp(ct^{1 - \lambda})} dt \\ & \leq \int_{x - \varepsilon x^\lambda}^{x} \difrac{(1 - \varepsilon)\exp(cx^{1 - \lambda}) - \exp(ct^{1 - \lambda})}{t^\lambda \exp(ct^{1 - \lambda})} dt \\ & = (1 - \varepsilon) \exp(cx^{1 - \lambda}) \int_{x - \varepsilon x^\lambda}^{x} \difrac{1}{t^\lambda \exp(ct^{1 - \lambda})} dt - \int_{x - \varepsilon x^\lambda}^{x} \difrac{1}{t^\lambda} dt \\ & = (1 - \varepsilon) \difrac{\exp[cx^{1 - \lambda} - c(x - \varepsilon x^{\lambda})^{1 - \lambda}] - 1}{c(1 - \lambda)} - \difrac{x^{1 - \lambda} - (x - \varepsilon x^{\lambda})^{1 - \lambda}}{1 - \lambda} \\ & < (1 - \varepsilon) \left[\difrac{\exp[c \varepsilon (1 - \lambda)] - 1}{c(1 - \lambda)} + \delta\right] - \varepsilon \\ & = (1 - \varepsilon)\difrac{\exp[c \varepsilon (1 - \lambda)] - 1}{c(1 - \lambda)} - \varepsilon + (1 - \varepsilon)\delta \\ & < -\difrac{\varepsilon^2}{2} + (1 - \varepsilon) \difrac{\varepsilon^2}{6(1 - \varepsilon)} = -\difrac{\varepsilon^2}{3} < 0, \end{split} \end{equation*} for arbitrarily large $x$, by lemma \ref{lemmat2}. This contradicts the fact that the integral converges. \end{proof}

\begin{center} \underline{Step 3} \end{center}

Now the time to make the change of variables in the integral of proposition \ref{step2} has come. Define $$\begin{matrix}
g: & [0, \infty) & \rightarrow & [1, \infty)\\ & x & \mapsto & [1 + (1 - \lambda)x]^{1/(1 - \lambda)} \end{matrix}$$

It's easy to see that $g(0) = 1$, $g$ is strictly increasing on $[0, \infty)$, bijective and of class $\mathscr{C}^{\infty}$. This function $g$ will be our substitution. Making $t = g(x)$, we get $$\diint_{1}^{\infty} \difrac{W(t) - \exp(ct^{1 - \lambda})}{t^{\lambda}\exp(ct^{1 - \lambda})} dt = \diint_{1}^{\infty} \left(\difrac{W(t)}{\exp(ct^{1 - \lambda})} - 1\right) \difrac{dt}{t^{\lambda}} = \diint_{0}^{\infty} \difrac{W(g(x))}{\exp[cg(x)^{1 - \lambda}]} - 1 dx,$$

\noindent so that now our goal is to show that $$\diint_{0}^{\infty} \difrac{W(g(x))}{\exp[cg(x)^{1 - \lambda}]} - 1 dx$$

\noindent converges.

\begin{center} \underline{Step 4} \end{center}

In order to apply Newman's analytic theorem, we now need to show that $W(x)$ is bounded by $\exp(cx^{1 - \lambda})$. The proof, although a bit tricky, it is based on the following theorem of Selberg.

\begin{theorem} [Selberg] \label{selberg}

There exists a constant $S > 0$ such that $$\pi(x+y) - \pi(x) \leq \difrac{Sy}{\log(y)}, \qquad \ \forall \ x, \ y \geq 2.$$

\end{theorem}

\begin{proof} See \cite{Selberg}. \end{proof}

It is worth mentioning that Montgomery, in \cite{Montgomery} (see p. $34$), showed that one can take $S = 2$ in theorem \ref{selberg}, although knowing an explicit value of the constant $S$ is not necessary for our purposes.

\begin{proposition} \label{mystep3}

There exists a constant $K > 0$ such that $W(x) \leq K\exp(cx^{1 - \lambda})$, $\forall \ x \geq 1$.

\end{proposition}

\begin{proof} Let $x_0 \in \R$, $x_0 \geq 2$, such that $w|_{[x_0, \infty)}$ is increasing and $w(x_0) \geq w(\tilde x)$, for all $1 \leq \tilde x \leq x_0$.

Now, for each $x \gg 0$, let $J \in \N$ be the greatest positive integer such that $$x - (J + 1)x^{\lambda} \geq x_0,$$

\noindent and define, for each $j \in \{0, \ldots, J\}$, $$I_j = (x - jx^{\lambda} - x^{\lambda}, x - jx^{\lambda}].$$

Then, by theorem \ref{selberg}, for each $j \in \{0, \ldots, J\}$, one has $$\#(\mathbb{P} \cap I_j) = \pi(x - jx^{\lambda}) - \pi(x - jx^{\lambda} - x^{\lambda}) \leq \difrac{Sx^{\lambda}}{\lambda \log(x)},$$

\noindent which implies that, for each $j \in \{0, \ldots, J\}$, \begin{equation*} \begin{split} \disum_{p \in I_j} w(p) & \leq \difrac{Sx^{\lambda}}{\lambda \log(x)} w(x - jx^{\lambda}) \\ & = \difrac{Sc(1 - \lambda) x^{\lambda} \log(x - jx^{\lambda}) \exp[c(x - jx^{\lambda})^{1 - \lambda}]}{\lambda \log(x) (x - jx^{\lambda})^{\lambda}},\end{split} \end{equation*}

\noindent which implies, in turn, that, for each $j \in \{0, \ldots, J\}$, $$\difrac{\disum_{p \in I_j} w(p)}{\exp(cx^{1 - \lambda})} \leq \difrac{Sc(1 - \lambda)}{\lambda} \underbrace{\difrac{\log(x - jx^{\lambda})}{\log(x)}}_{\leq \ 1} \underbrace{\difrac{x^{\lambda}}{(x - jx^{\lambda})^{\lambda}}}_{\leq \ (1 + j)^{\lambda}} \underbrace{\difrac{\exp[c(x - jx^{\lambda})^{1 - \lambda}]}{\exp(cx^{1 - \lambda})}}_{\leq \ \exp[-c(1 - \lambda)j]} \leq \difrac{Sc(1 - \lambda)}{\lambda} \difrac{(1 + j)^{\lambda}}{\exp[c(1 - \lambda)j]},$$

\noindent by lemma \ref{lemmah2}. Also note that $x - (J + 1)x^{\lambda} < x_0 + x^{\lambda}$, which implies that \begin{equation*} \begin{split} \difrac{\disum_{p \leq x - (J + 1)x^{\lambda}} w(p)}{\exp(cx^{1 - \lambda})} & \leq \difrac{\disum_{p \leq x_0 +x^{\lambda}} w(p)}{\exp(cx^{1 - \lambda})} \\ & \leq \difrac{(x_0 + x^{\lambda})w(x_0 + x^{\lambda})}{\exp(cx^{1 - \lambda})}, \end{split} \end{equation*}

\noindent and since that $$\difrac{(x_0 + x^{\lambda})w(x_0 + x^{\lambda})}{\exp(cx^{1 - \lambda})} = \difrac{(x_0 + x^{\lambda})c(1 - \lambda) \log(x_0 + x^{\lambda}) \exp[c(x_0 + x^{\lambda})^{1 - \lambda}]}{(x_0 + x^{\lambda})^{\lambda} \exp(cx^{1 - \lambda})} \xrightarrow[x \to \infty]{} 0,$$

\noindent there exists $k > 0$ (independent of $x$) such that $$\difrac{\disum_{p \leq x - (J + 1)x^{\lambda}} w(p)}{\exp(cx^{1 - \lambda})} \leq k.$$

Putting all this information together, \begin{equation*} \begin{split} \difrac{W(x)}{\exp(cx^{1 - \lambda})} & = \difrac{\disum_{p \leq x} w(p)}{\exp(cx^{1 - \lambda})} \\ & = \difrac{\disum_{p \leq x - (J + 1)x^{\lambda}} w(p) + \disum_{j = 0}^{J} \disum_{p \in I_j} w(p)}{\exp(cx^{1 - \lambda})} \\ & = \difrac{\disum_{p \leq x - (J + 1)x^{\lambda}} w(p)}{\exp(cx^{1 - \lambda})} + \disum_{j = 0}^{J} \difrac{\disum_{p \in I_j} w(p)}{\exp(cx^{1 - \lambda})} \\ & \leq k + \disum_{j = 0}^{J} \difrac{Sc(1 - \lambda)}{\lambda} \difrac{(1 + j)^{\lambda}}{\exp[c(1 - \lambda)j]} \\ & < k + \difrac{Sc(1 - \lambda)}{\lambda} \disum_{j = 0}^{\infty} \difrac{(1 + j)^{\lambda}}{\exp[c(1 - \lambda)j]} < \infty, \end{split} \end{equation*}

\noindent what concludes the proof. \end{proof}

\vspace{2mm} \begin{center} \underline{Step 5} \end{center}

\vspace{3mm} This may be a little surprising, but we don't need to change anything!

\newpage

\begin{center} \underline{Step 6} \end{center}

\begin{defi}
Define $$\begin{matrix}
\Psi: & \{s \in \C; \ \Re(s) > 0\} & \rightarrow & \C \\ & s & \mapsto & \disum_p \difrac{\log(p)}{p^{\lambda}\exp\left[s \left(\difrac{p^{1 - \lambda} - 1}{1 - \lambda}\right)\right]}. \end{matrix}$$
\end{defi}

\begin{proposition}

The Laplace transform of the function $\difrac{W(g(x))}{\exp[cg(x)^{1 - \lambda}]} -1$ is $$\mathscr{L}(s) = \difrac{c(1 - \lambda)}{c(1 - \lambda) + s}\Psi(s) - \difrac{1}{s}.$$

\end{proposition}

\begin{proof} For each $n \in \N$, let $q_n > 0$ such that $g(q_n) = p_n$. Thus, for each $s \in \C$ such that $\Re(s) > 0$, $$\mathscr{L}(s) = \diint_{0}^{\infty} \left(\difrac{W(g(x))}{\exp[cg(x)^{1 - \lambda}]} -1\right) \exp(-sx) dx = \diint_{0}^{\infty} \difrac{W(g(x))}{\exp[cg(x)^{1 - \lambda} + sx]} dx - \difrac{1}{s}.$$

To calculate the first integral, we will use the known trick of splitting it into intervals over which $W$ is constant. Remembering that $W(x) = 0$ if $x < 2$, we get \\ $\diint_{0}^{\infty} \difrac{W(g(x))}{\exp[cg(x)^{1 - \lambda} + sx]} dx =$ \begin{equation*} \begin{split} {} & = \disum_{n = 1}^{\infty} \diint_{q_n}^{q_{n+1}} \difrac{W(g(x))}{\exp[cg(x)^{1 - \lambda} + sx]} dx \\ & = \disum_{n = 1}^{\infty} W(p_n) \diint_{q_n}^{q_{n+1}} \difrac{1}{\exp[cg(x)^{1 - \lambda} + sx]} dx \\ & = \disum_{n = 1}^{\infty} W(p_n) \diint_{q_n}^{q_{n+1}} \difrac{1}{\exp[c + c(1 - \lambda)x + sx]} dx \\ & = \difrac{1}{c(1 - \lambda) + s} \disum_{n = 1}^{\infty} W(p_n) \cdot \difrac{1}{\exp[c + c(1 - \lambda)x + sx]} \bigg|_{q_{n+1}}^{q_n} \\ & = \difrac{\exp[s/(1 - \lambda)]}{c(1 - \lambda) + s} \disum_{n = 1}^{\infty} W(p_n) \left[\difrac{1}{\exp[cp_n^{1 - \lambda} + sp_n^{1 - \lambda}/(1 - \lambda)]} - \difrac{1}{\exp[cp_{n + 1}^{1 - \lambda} + sp_{n + 1}^{1 - \lambda}/(1 - \lambda)]} \right] \\ & = \difrac{\exp[s/(1 - \lambda)]}{c(1 - \lambda) + s}\left[\difrac{W(p_1)}{\exp[cp_1^{1 - \lambda} + sp_1^{1 - \lambda}/(1 - \lambda)]} + \disum_{n = 2}^{\infty} \difrac{W(p_n) - W(p_{n - 1})}{\exp[cp_n^{1 - \lambda} + sp_n^{1 - \lambda}/(1 - \lambda)]} \right] \\ & = \difrac{\exp[s/(1 - \lambda)]}{c(1 - \lambda) + s}\disum_p \difrac{w(p)}{\exp[cp^{1 - \lambda} + sp^{1 - \lambda}/(1 - \lambda)]} \\ & = \difrac{c(1 - \lambda)}{c(1 - \lambda) + s} \disum_p \difrac{\log(p)}{p^{\lambda}\exp\left[s \left(\difrac{p^{1 - \lambda} - 1}{1 - \lambda}\right)\right]} \\ & = \difrac{c(1 - \lambda)}{c(1 - \lambda) + s}\Psi(s), \end{split} \end{equation*}

\noindent what concludes the proof. \end{proof}

\begin{center} \underline{Step 7} \end{center}

\begin{proposition} \label{firsttweak}

In order to prove that formula \ref{eq1.2} is true, it is sufficient to show that $$\Psi(s) - \difrac{1}{s},$$

\noindent initially defined for $\Re(s) > 0$, extends analytically to $\Re(s) \geq 0$.

\end{proposition}

\begin{proof} Algebraic manipulation only: just see that $$\difrac{c(1 - \lambda)}{c(1 - \lambda) + s}\Psi(s) - \difrac{1}{s} = \difrac{c(1 - \lambda)}{c(1 - \lambda) + s}\left[\Psi(s) - \difrac{1}{s} - \difrac{1}{c(1 - \lambda)}\right]$$

\noindent and apply Newman's analytic theorem. \end{proof}

If we wanted to continue to closely follow in Newman's footsteps, we should now use some kind of algebraic identity (of the type described in Step $8$) to finish the proof of the formulas \ref{eq1.1} and \ref{eq1.2}. But this would lead us to a problem of commuting an infinite double series (explained in section \ref{section4}) that I honestly don't know how to work around. Even so, we can complete the proof of the formulas \ref{eq1.1} and \ref{eq1.2}. How?

\begin{center} \underline{Step 8} \end{center}

For $s \in \C$ such that $\Re(s) > 0$, define the functions $$\Xi(s) = \disum_p \difrac{(1 - \lambda) \log(p)}{p^\lambda \exp[s(p^{1 - \lambda} - 1)]} \qquad \qquad \text{and} \qquad \qquad \tau(s) = \disum_p \difrac{(1 - \lambda) \log(p)}{p^\lambda \exp(sp^{1 - \lambda})}.$$

Obviously these functions are analytic on $\Re(s) > 0$.

\begin{proposition}

The function $\Psi(s) - \difrac{1}{s}$ extends analytically to $\Re(s) \geq 0$ if, and only if, the function $\Xi(s) - \difrac{1}{s}$ extends analytically to $\Re(s) \geq 0$.

\end{proposition}

\begin{proof} $(\Rightarrow)$ By assumption, there is an analytic function $s \mapsto F(s)$, defined on an open set containing $\Re(s) \geq 0$, such that $$\Psi(s) - \difrac{1}{s} = F(s),$$

\noindent for all $s \in \C$ such that $\Re(s) > 0$. As $\lambda \in \R$, $0 < \lambda < 1$, then $$\Psi((1 - \lambda)s) - \difrac{1}{(1 - \lambda)s} = F((1 - \lambda)s),$$

\noindent for all $s \in \C$ such that $\Re(s) > 0$. Multiplying by $1 - \lambda$, one gets $$(1 - \lambda)\Psi((1 - \lambda)s) - \difrac{1}{s} = (1 - \lambda)F((1 - \lambda)s),$$

\noindent for all $s \in \C$ such that $\Re(s) > 0$, that is $$\Xi(s) - \difrac{1}{s} = (1 - \lambda)F((1 - \lambda)s),$$

\noindent for all $s \in \C$ such that $\Re(s) > 0$. This shows that $\Xi(s) - \difrac{1}{s}$ extends analytically to $\Re(s) \geq 0$.

\noindent $(\Leftarrow)$ By assumption, there is an analytic function $s \mapsto G(s)$, defined on an open set containing $\Re(s) \geq 0$, such that

$$\Xi(s) - \difrac{1}{s} = G(s),$$

\noindent for all $s \in \C$ such that $\Re(s) > 0$. As $\lambda \in \R$, $0 < \lambda < 1$, then

$$\Xi\left(\difrac{s}{1 - \lambda}\right) - \difrac{1 - \lambda}{s} = G\left(\difrac{s}{1 - \lambda}\right),$$

\noindent for all $s \in \C$ such that $\Re(s) > 0$. Dividing by $1 - \lambda$, one gets

$$\difrac{1}{1 - \lambda} \Xi\left(\difrac{s}{1 - \lambda}\right) - \difrac{1}{s} = \difrac{1}{1 - \lambda}G\left(\difrac{s}{1 - \lambda}\right),$$

\noindent for all $s \in \C$ such that $\Re(s) > 0$, that is,

$$\Psi(s) - \difrac{1}{s} = \difrac{1}{1 - \lambda}G\left(\difrac{s}{1 - \lambda}\right),$$

\noindent for all $s \in \C$ such that $\Re(s) > 0$. This shows that $\Psi(s) - \difrac{1}{s}$ extends analytically to $\Re(s) \geq 0$. \end{proof}

\vspace{5mm} \begin{center} \underline{Step 9} \end{center}

\vspace{5mm} In turn, the same happens with the function $\tau(s)$.

\vspace{3mm} \begin{proposition}

The function $\Xi(s) - \difrac{1}{s}$ extends analytically to $\Re(s) \geq 0$ if, and only if, the function $\tau(s) - \difrac{1}{s}$ extends analytically to $\Re(s) \geq 0$.

\end{proposition}

\vspace{3mm} \begin{proof} Note that for all $s \in \C$ such that $\Re(s) > 0$ holds

\begin{equation*} \begin{split} \Xi(s) - \difrac{1}{s} & = \disum_p \difrac{(1 - \lambda)\log(p)}{p^\lambda \exp[s(p^{1 - \lambda} - 1)]} - \difrac{1}{s} \\ & = \disum_p \difrac{(1 - \lambda)\log(p)}{p^\lambda \exp(sp^{1 - \lambda}) \exp(-s)} - \difrac{1}{s} \\ & = \exp(s) \tau(s) - \difrac{1}{s} \\ & = \exp(s) \left[\tau(s) - \difrac{1}{s} - \left(\difrac{1}{s \exp(s)} - \difrac{1}{s}\right)\right]. \end{split} \end{equation*}

\vspace{3mm} But since the function $s \mapsto \difrac{1}{s \exp(s)} - \difrac{1}{s}$ extends analytically to $\C$, we are done. \end{proof}

\newpage

\begin{center} \underline{Step 10} \end{center}

This is our last step, and the point where all our work culminates. First, for $s \in \C$ such that $\Re(s) > 0$, define the function $$T(s) = -\disum_p \difrac{(1 - \lambda) \log(p)}{p^{3 - 2\lambda} \exp(sp^{1 - \lambda})}.$$

Obviously this function is analytic on $\Re(s) > 0$. Now, consider the following well-known theorems.

\begin{theorem} \label{wellknowntheorem}

The function $$\Phi(s + 1) - \difrac{1}{s} = \disum_p \difrac{\log(p)}{p^{s + 1}} - \difrac{1}{s},$$ initially defined on $\Re(s) > 0$, extends to a meromorphic function on $\Re(s) > -\difrac{1}{2}$ that is analytic on $\Re(s) \geq 0$.

\end{theorem}

\begin{proof} See \cite{Sutherland}. \end{proof}

Denote by $\Omega(s)$ the extension referred in the theorem \ref{wellknowntheorem}.

\begin{theorem} \label{wellknowntheorem2}

There exists a constant $b > 0$ such that the Riemann zeta function has no zeros in the region $$\left\{s = \sigma + it \in \C; \ \sigma > 1 - \difrac{b}{\log(2+|t|)}\right\}.$$

\noindent Consequently, there exists a constant $a > 0$ such that $\Omega(s)$ is analytic on the open set $$\mathcal{U} = \left\{s = \sigma + it \in \C; \ \sigma > - \difrac{a}{\log(2+|t|)}\right\}.$$

\noindent Furthermore, within the open neighborhood of the imaginary axis defined by $$\mathcal{V} = \left\{s = \sigma + it \in \C; \ |\sigma| < \difrac{a}{\log(2+|t|)}\right\},$$ the function $\Omega(s)$ satisfies the growth estimate $$\Omega(s) = O(\log(2 + |t|)).$$ 

\end{theorem}

\begin{proof} See \cite{Davenport}. \end{proof}

The central strategy for completing the proofs of formulas \ref{eq1.1} and \ref{eq1.2} involves computing the Mellin transform of $T(s)$. By rearranging the resulting expression into a suitable form and applying the inverse Mellin transform, we shall establish the analytic continuation of the function $\tau(s) - \difrac{1}{s}$.

\begin{proposition}

The function $\tau(s) - \difrac{1}{s}$, initially defined for $\Re(s) > 0$, extends analytically to $\Re(s) \geq 0$.

\end{proposition}

\begin{proof} As mentioned above, calculating the Mellin transform of the function $T(s)$ for $\Re(s) > 0$, we obtain \begin{equation*} \begin{split} \mathscr{M}(z) & = \mathscr{M}\left(-\disum_p \difrac{(1 - \lambda) \log(p)}{p^{3 - 2\lambda} \exp(sp^{1 - \lambda})}\right) \\ & = -\disum_p \difrac{(1 - \lambda) \log(p)}{p^{3 - 2\lambda}} \mathscr{M}[\exp(-sp^{1 - \lambda})] \\ & = -\disum_p \difrac{(1 - \lambda) \log(p)}{p^{3 - 2\lambda}} \cdot \difrac{\Gamma(z)}{p^{(1 - \lambda)z}} \\ & = -(1 - \lambda) \Gamma(z) \disum_p \difrac{\log(p)}{p^{1 + (1 - \lambda)(z + 2)}} \\ & = -(1 - \lambda) \Gamma(z) \Phi(1 + (1 - \lambda)(z + 2)). \end{split} \end{equation*}

Note that this function is meromorphic on $\Re(z) > -2$, and has two simple poles: one at $z = 0$ and another at $z = -1$; not because of the function $\Phi(1 + (1 - \lambda)(z + 2))$, which is analytic on $\Re(z) > -2$, but because of the function $\Gamma(z)$ with its two simple poles at $z = 0$ and $z = -1$. Using the information from theorem \ref{wellknowntheorem}, we can write \begin{equation*} \begin{split} \mathscr{M}(z) & = -(1 - \lambda) \Gamma(z) \left[\difrac{1}{(1 - \lambda)(z + 2)} + \Omega((1 - \lambda)(z + 2))\right] \\ & = -\difrac{\Gamma(z)}{z + 2} - (1 - \lambda) \Gamma(z) \Omega((1 - \lambda)(z + 2)). \end{split} \end{equation*}

Now, for $\Re(z) > -2$, notice that this function has exponential decay along vertical lines, because $\Gamma(z)$ does, and that the function $\Omega((1 - \lambda)(z + 2))$ doesn't change that . Therefore, for $\Re(z) > -2$, we can get back the function $T(s)$ by computing the inverse Mellin transform of $\mathscr{M}(z)$, and write \begin{equation*} \begin{split} T(s) & = \mathscr{M}^{-1}(\mathscr{M}(z)) \\ & = \mathscr{M}^{-1}\left(-\difrac{\Gamma(z)}{z + 2} - (1 - \lambda) \Gamma(z) \Omega((1 - \lambda)(z + 2))\right) \\ & = -\mathscr{M}^{-1}\left(\difrac{\Gamma(z)}{z + 2}\right) - (1 - \lambda) \mathscr{M}^{-1}[\Gamma(z) \Omega((1 - \lambda)(z + 2))] \\ & = -\difrac{s^2 \cdot \Gamma(0,s) + \exp(-s) \cdot (1-s)}{2} - (1 - \lambda) \mathscr{M}^{-1}[\Gamma(z) \Omega((1 - \lambda)(z + 2))] \end{split} \end{equation*}

\noindent for all $s \in \C$ such that $\Re(s) > 0$. Differentiating three times, one gets \begin{equation*} \begin{split} \tau(s) & = T'''(s) \\ & = \difrac{d^3}{ds^3} \left(-\difrac{s^2 \cdot \Gamma(0,s) + \exp(-s) \cdot (1-s)}{2}\right) - (1 - \lambda) \difrac{d^3}{ds^3} \left(\mathscr{M}^{-1}[\Gamma(z) \Omega((1 - \lambda)(z + 2))]\right) \\ & = \difrac{\exp(-s)}{s} - (1 - \lambda) \difrac{d^3}{ds^3} \left(\mathscr{M}^{-1}[\Gamma(z) \Omega((1 - \lambda)(z + 2))]\right). \end{split} \end{equation*}

Subtracting $\difrac{1}{s}$ on either side of this equation, we get $$\tau(s) - \difrac{1}{s} = \difrac{\exp(-s) - 1}{s} - (1 - \lambda) \difrac{d^3}{ds^3} \left(\mathscr{M}^{-1}[\Gamma(z) \Omega((1 - \lambda)(z + 2))]\right).$$

But we have already seen that the function $$s \mapsto \difrac{\exp(-s) - 1}{s}$$ extends analytically to $\C$. The key observation is that the function $$\mathscr{M}^{-1}[\Gamma(z) \Omega((1 - \lambda)(z + 2))] = \difrac{1}{2\pi i} \diint_{c - i\infty}^{c + i\infty} s^{-z} \cdot \Gamma(z) \cdot \Omega((1-\lambda)(z+2)) dz,$$ initially defined and analytic on $\Re(s) > 0$, can be analytically extended to $\Re(s) \geq 0$ via a deformation of the contour $\mathcal{C}$ of the integral for the points that are on (and close to the left of) the line $\Re(z) = -2$. Consequently, the function $$\difrac{d^3}{ds^3} \left(\mathscr{M}^{-1}[\Gamma(z) \Omega((1 - \lambda)(z + 2))]\right),$$ being the third derivative of an analytic function in $\Re(s) \geq 0$, will itself be analytic in $\Re(s) \geq 0$, thus concluding the proof of this proposition.

With this in mind, by making the change of variables $z + 2 \leadsto z$ and using the classic property of the $\Gamma$ function, we can write \begin{equation*} \begin{split} \mathscr{M}^{-1}[\Gamma(z) \Omega((1 - \lambda)(z + 2))] & = \difrac{1}{2\pi i} \diint_{c - i\infty}^{c + i\infty} s^{-z} \cdot \Gamma(z) \cdot \Omega((1-\lambda)(z+2)) dz \\ & = \difrac{1}{2\pi i} \diint_{c - i\infty}^{c + i\infty} s^{-(z - 2)} \cdot \Gamma(z-2) \cdot \Omega((1 - \lambda)z) dz \\ & = \difrac{s^2}{2\pi i} \diint_{c - i\infty}^{c + i\infty} \difrac{s^{-z} \cdot \Gamma(z) \cdot \Omega((1 - \lambda)z)}{(z - 1) \cdot (z - 2)} dz. \end{split} \end{equation*}

For $\Re(z) > 0$, this function is clearly analytic. Now, for the points on the line $\Re(z) = 0$ (and points near the left of this line), we will use the information from theorem \ref{wellknowntheorem2}. Write $$\Upsilon(s, z) = \difrac{s^{-z} \cdot \Gamma(z) \cdot \Omega((1 - \lambda)z)}{(z - 1) \cdot (z - 2)}$$ to simplify the notation. By the residue theorem, we can write \begin{equation*} \begin{split} \diint_{c - i\infty}^{c + i\infty} \Upsilon(s, z) dz = 2\pi i [\Res(\Upsilon, 2) + \Res(\Upsilon, 1) + \Res(\Upsilon, 0)] + \diint_{\mathcal{C}} \Upsilon(s, z) dz, \end{split} \end{equation*} where $\mathcal{C}$ is the contour given by $$\mathcal{C} = \left\{\sigma(t) + it; \ t \in \R, \ \sigma(t) = \frac{-a}{2\log(2 + |t|)}\right\}.$$

Note that contour $\mathcal{C}$ lies within the set $\mathcal{V}$ of theorem \ref{wellknowntheorem2}. Calculating the residues, one has: 

\noindent $\bullet \ \Res(\Upsilon, 2) = \dilim_{z \to 2} (z - 2) \cdot \Upsilon(s, z) = \difrac{s^{-2} \cdot \Gamma(2) \cdot \Omega((1 - \lambda) \cdot 2)}{2 - 1} = \difrac{\Omega(2 \cdot (1 - \lambda))}{s^2}$

\noindent $\bullet \ \Res(\Upsilon, 1) = \dilim_{z \to 1} (z - 1) \cdot \Upsilon(s, z) = \difrac{s^{-1} \cdot \Gamma(1) \cdot \Omega((1 - \lambda) \cdot 1)}{1 - 2} = -\difrac{\Omega(1 - \lambda)}{s}$

\noindent $\bullet \ \Res(\Upsilon, 0) = \dilim_{z \to 0} z \cdot \Upsilon(s, z) = \difrac{s^0 \cdot \Omega((1 - \lambda) \cdot 0)}{(-1) \cdot (-2)} \cdot \dilim_{z \to 0} z \cdot \Gamma(z) = \difrac{\Omega(0)}{2}$

\noindent so that \begin{equation*} \begin{split} \mathscr{M}^{-1}[\Gamma(z) \Omega((1 - \lambda)(z + 2))] & = \difrac{s^2}{2\pi i} \diint_{c - i\infty}^{c + i\infty} \Upsilon(s, z) dz \\ & = \Omega(2 \cdot (1 - \lambda)) - s \cdot \Omega(1 - \lambda) + s^2 \cdot \difrac{\Omega(0)}{2} + \difrac{s^2}{2\pi i} \diint_{\mathcal{C}} \Upsilon(s, z) dz. \end{split} \end{equation*}

Thus, it suffices to show that the remaining integral is absolutely convergent for $s \in \mathcal{V}$. For each of these $s$, we have:

\noindent $\bullet \ |s^{-z}| = |s^{-\sigma(t) - it}| = |s|^{-\sigma(t)} \cdot \exp[t \cdot \arg(s)]$

\noindent $\bullet \ |\Gamma(z)| = |\Gamma(\sigma(t) + it)| \sim \sqrt{2\pi} \cdot |t|^{\sigma(t) - 1/2} \cdot \exp\left(-\difrac{\pi}{2} \cdot |t|\right)$ as $|t| \to \infty$ (uniformly for $\sigma$ in compact 

\noindent intervals, by Stirling's approximation)

\noindent $\bullet \ \Omega((1 - \lambda)z) \sim \log(|t|)$ \ \ \ \ \ (by theorem \ref{wellknowntheorem2})

\noindent $\bullet |(z - 1) \cdot (z - 2)| \sim t^2 + \difrac{5}{2}$

With all this in hand, we finally have \begin{equation*} \begin{split} |\Upsilon(s, z)| & = \difrac{|s^{-z}| \cdot |\Gamma(z)| \cdot |\Omega((1 - \lambda)z)|}{|(z - 1) \cdot (z - 2)|} \\ & \sim \difrac{\sqrt{2\pi} \cdot |s|^{-\sigma(t)} \cdot \exp\left[t \cdot \arg(s) - \difrac{\pi}{2} \cdot |t|\right] \cdot |t|^{-\sigma(t) - 1/2} \cdot \log(|t|)}{t^2 + 5/2} \\ & \xrightarrow[t \to \infty]{} \sqrt{2\pi} \cdot \exp\left(\difrac{a}{2}\right) \cdot \difrac{|t|^{-1/2} \cdot \log(|t|)}{t^2 + 5/2}, \end{split} \end{equation*} ensuring that the integral $$\diint_{-\infty}^{\infty} |\Upsilon(s, z)| dt \sim \diint_{-\infty}^{\infty} \difrac{|t|^{-1/2} \cdot \log(|t|)}{t^2 + 5/2} dt$$ converges. Furthermore, the same applies to the derivative of $\Upsilon(s, z)$ with respect to $s$: \begin{equation*} \begin{split} \Bigg|\difrac{d}{ds} \Upsilon(s, z)\Bigg| & = \Bigg| \difrac{d}{ds} \difrac{s^{-z} \cdot \Gamma(z) \cdot \Omega((1 - \lambda)z)}{(z - 1) \cdot (z - 2)} \Bigg| \\ & = \Bigg| \difrac{-z \cdot s^{- z - 1} \cdot \Gamma(z) \cdot \Omega((1 - \lambda)z)}{(z - 1) \cdot (z - 2)}\Bigg| \\ & = |-z| \cdot |s^{-1}| \cdot |\Upsilon(s, z)| \\ & \sim \difrac{|t|}{|s|} \cdot \sqrt{2\pi} \cdot \exp\left(\difrac{a}{2}\right) \cdot \difrac{|t|^{-1/2} \cdot \log(|t|)}{t^2 + 5/2} \\ & = \difrac{\sqrt{2\pi}}{|s|} \cdot \exp\left(\difrac{a}{2}\right) \cdot \difrac{|t|^{1/2} \cdot \log(|t|)}{t^2 + 5/2}, \end{split} \end{equation*} ensuring that the integral $$\diint_{-\infty}^{\infty} \Bigg|\difrac{d}{ds} \Upsilon(s, z)\Bigg| dt \sim \diint_{-\infty}^{\infty} \difrac{|t|^{1/2} \cdot \log(|t|)}{t^2 + 5/2} dt$$ also converges.

Since the integral $\diint_{\mathcal{C}} \Upsilon(s, z) dz$ and its derivative (with respect to $s$) converge absolutely and uniformly for $s$ in any compact subset of the half-plane $\Re(s) \geq 0$ (noting that the $s^2$ factor at the front compensates for the $s^{-1}$ term in the derivative as $s \to 0$), it follows from the differentiation theorem under the integration sign that the function$$s \mapsto \difrac{s^2}{2\pi i} \int_{\mathcal{C}} \Upsilon(s, z) dz $$is analytic in an open neighborhood of the half-plane $\Re(s) \geq 0$. \end{proof}

\begin{theorem} \label{mainresult}

Let $\lambda \in \R$, $0 < \lambda < 1$. Then $\pi(x + x^\lambda) - \pi(x) \sim \difrac{x^\lambda}{\log(x)}$.

\end{theorem}

\begin{proof} Since $\tau(s) - 1/s$ extends analytically to $\Re(s) \geq 0$, then $\Psi(s) - 1/s$ extends analytically to $\Re(s) \geq 0$, thus proving the validity of formulas \ref{eq1.1} and \ref{eq1.2}. \end{proof}

It is worth remarking why the third primitive (leading to $T''' = \tau$) is required for this argument, rather than the second. Had we defined an auxiliary function $Q(s)$ such that $Q''(s) = \tau(s)$, the corresponding inverse Mellin transform would involve the integrand: \begin{equation*} \tilde{\Upsilon}(s, z) = \frac{s^{-z} \cdot \Gamma(z) \cdot \Omega((1 - \lambda)z)}{z - 1}. \end{equation*}

On the critical boundary $\Re(s) = 0$, the exponential decay of the Gamma function exactly balances the growth of the term $s^{-z}$, as $$\bigg|\exp\left(t \arg(s) - \frac{\pi}{2}|t|\right)\bigg| = 1 \qquad \qquad \text{for } \arg(s) = \pm \difrac{\pi}{2}.$$

In this regime, Stirling's approximation yields $$\Big|\tilde{\Upsilon}(s, z)\Big| \sim |t|^{-3/2} \log(|t|) \qquad \qquad \qquad \text{as } |t| \to \infty.$$

While this is absolutely integrable, a formal recovery of $\tau(s)$ via differentiation under the integral sign would introduce a factor of $-z$ in the numerator. The resulting integrand for $\tau(s)$ would then satisfy: \begin{equation*} \left| \frac{-z \cdot s^{-z} \cdot \Gamma(z) \cdot \Omega((1 - \lambda)z)}{z-1} \right| \sim |t|^{-1/2} \log(|t|), \qquad \qquad \text{as } |t| \to \infty. \end{equation*}

Since $t^{-1/2} \log(|t|)$ is not integrable on $[1, \infty)$, the integral defining $\tau(s)$ would fail to converge absolutely on the boundary. By contrast, employing the third primitive introduces a quadratic denominator $$(z-1) \cdot (z-2).$$

This ensures that even after differentiating to recover $\tau(s)$, the integrand remains $$O\left(|t|^{-3/2} \log(|t|)\right).$$

Combined with the fact that the deformed contour $\mathcal{C}$ is bounded away from the poles of $\Gamma(z)$ and $\Omega((1 - \lambda)z)$, this absolute integrability ensures that $\tau(s)$ is well-defined and analytic in an open neighborhood of the half-plane $\Re(s) \geq 0$.

\section{Some technical lemmas} \label{section3}

\begin{lemma} \label{lemma1}

Let $\varepsilon \in \R$. Then $\dilim_{x \to \infty} (x + \varepsilon x^\lambda)^{1 - \lambda} - x^{1 - \lambda} = \varepsilon (1 - \lambda)$.

\end{lemma}

\begin{proof} $\dilim_{x \to \infty} (x + \varepsilon x^\lambda)^{1 - \lambda} - x^{1 - \lambda} = \dilim_{x \to \infty} \difrac{(1 + \varepsilon x^{\lambda - 1})^{1 - \lambda} - 1}{x^{\lambda - 1}} \overset{\text{LH}}{=} \dilim_{x \to \infty} \difrac{\varepsilon (1 - \lambda)}{(1 + \varepsilon x^{\lambda - 1})^{\lambda}} = \varepsilon (1 - \lambda)$. \end{proof}

\begin{lemma} \label{integrallimsup}

Let $\alpha \in \R$ and $u = \difrac{\alpha x^{\lambda}}{\log(x)}$. Then $\dilim_{x \to \infty} \difrac{\diint_{x}^{x + u} w(t) dt}{\exp(cx^{1 - \lambda})} = \alpha c(1 - \lambda)$.

\end{lemma}

\begin{proof} $\difrac{\diint_{x}^{x + u} w(t) dt}{\exp(cx^{1 - \lambda})} \overset{\text{LH}}{=} \difrac{w(x + u)(1 + u') - w(x)}{\exp(cx^{1 - \lambda})c(1 - \lambda)x^{-\lambda}}$ \\ \null \\ \null $\qquad \ \ \ \  = \log(x)\left[\difrac{\log(x + u)}{\log(x)} \difrac{x^{\lambda}}{(x + u)^{\lambda}} \difrac{\exp[c(x + u)^{1 - \lambda}]}{\exp(cx^{1 - \lambda})} - 1\right] +$ \\ \null \\ \null $\qquad \qquad \qquad \qquad + \log(x + u) \difrac{\exp[c(x + u)^{1 - \lambda}]}{\exp(cx^{1 - \lambda})} \difrac{x^{\lambda}}{(x + u)^{\lambda}} \difrac{\alpha x^{\lambda - 1}[\log(x) - 1]}{\log(x)^2}$ \\ \null \\ \null $\qquad \ \ \ \  = \log(x)\left[\left(1 + o\left(\difrac{1}{x^{1 - \lambda}}\right)\right) \left(1 + o\left(\difrac{1}{x^{1 - \lambda}}\right)\right) \left(1 + \difrac{\alpha c(1 - \lambda)}{\log(x)} +  O\left(\difrac{1}{\log(x)^2}\right)\right) - 1\right] +$ \\ \null \\ \null $\qquad \qquad \qquad \qquad + \log(x + u) \difrac{\exp[c(x + u)^{1 - \lambda}]}{\exp(cx^{1 - \lambda})} \difrac{x^{\lambda}}{(x + u)^{\lambda}} \difrac{\alpha x^{\lambda - 1}[\log(x) - 1]}{\log(x)^2}$ \\ \null \\ \null $\qquad \ \ \ \ \xrightarrow[x \to \infty]{} \alpha c(1 - \lambda).$ \end{proof}

\begin{lemma} \label{integralliminf}

Let $\alpha \in \R$ and $v = \difrac{\alpha x^{\lambda}}{\log(x)}$. Then $\dilim_{x \to \infty} \difrac{\diint_{x + x^{\lambda} - v}^{x + x^{\lambda}} w(t) dt}{\exp(cx^{1 - \lambda})} = \alpha c(1 - \lambda)\exp[c(1 - \lambda)]$.

\end{lemma}

\begin{proof} $\difrac{\diint_{x + x^{\lambda} - v}^{x + x^{\lambda}} w(t) dt}{\exp(cx^{1 - \lambda})} \overset{\text{LH}}{=} \difrac{w(x + x^{\lambda})(1 + \lambda x^{\lambda - 1}) - w(x + x^{\lambda} - v)(1 + \lambda x^{\lambda - 1} - v')}{\exp(cx^{1 - \lambda})c(1 - \lambda)x^{-\lambda}}$ \\ \null \\ \null $\qquad \qquad = \log(x)\left[\difrac{\log(x + x^{\lambda})}{\log(x)} - \difrac{\log(x + x^{\lambda} - v)}{\log(x)} \difrac{(x + x^{\lambda})^{\lambda}}{(x + x^{\lambda} - v)^{\lambda}} \difrac{\exp[c(x + x^{\lambda} - v)^{1 - \lambda}]}{\exp[c(x + x^{\lambda})^{1 - \lambda}]}\right] \times$ \\ \null \\ \null $\qquad \qquad \qquad \qquad \qquad \qquad \qquad \qquad \qquad \qquad \qquad \qquad \qquad \qquad \ \ \ \ \ \times \difrac{x^{\lambda}}{(x + x^{\lambda})^{\lambda}}\difrac{\exp[c(x + x^{\lambda})^{1 - \lambda}]}{\exp(cx^{1 - \lambda})} +$ \\ \null \\ \null $\qquad \qquad \qquad \qquad + \log(x + x^{\lambda}) \difrac{\exp[c(x + x^{\lambda})^{1 - \lambda}]}{\exp(cx^{1 - \lambda})} \difrac{x^{\lambda}}{(x + x^{\lambda})^{\lambda}} \difrac{\lambda}{x^{1 - \lambda}} +$ \\ \null \\ \null $\qquad \qquad \qquad \qquad - \log(x + x^{\lambda} - v) \difrac{\exp[c(x + x^{\lambda} - v)^{1 - \lambda}]}{\exp(cx^{1 - \lambda})} \difrac{x^{\lambda}}{(x + x^{\lambda} - v)^{\lambda}} \difrac{\lambda}{x^{1 - \lambda}} +$ \\ \null \\ \null $\qquad \qquad \qquad \qquad + \log(x + x^{\lambda} - v) \difrac{\exp[c(x + x^{\lambda} - v)^{1 - \lambda}]}{\exp(cx^{1 - \lambda})} \difrac{x^{\lambda}}{(x + x^{\lambda} - v)^{\lambda}} \difrac{\alpha x^{\lambda - 1}[\lambda \log(x) - 1]}{\log(x)^2}$ \\ \null \\ \null $\qquad \qquad = \log(x)\left[\left(1 + o\left(\difrac{1}{x^{1 - \lambda}}\right)\right) - \left(1 + o\left(\difrac{1}{x^{1 - \lambda}}\right)\right)\left(1 + o\left(\difrac{1}{x^{1 - \lambda}}\right)\right)\times\right.$ \\ \null \\ \null $\qquad \qquad \qquad \qquad \qquad \qquad \ \ \ \ \ \ \ \times \left.\left(1 - \difrac{\alpha c(1 - \lambda)}{\log(x)} + O\left(\difrac{1}{\log(x)^2}\right)\right)\right] \difrac{x^{\lambda}}{(x + x^{\lambda})^{\lambda}}\difrac{\exp[c(x + x^{\lambda})^{1 - \lambda}]}{\exp(cx^{1 - \lambda})} +$ \\ \null \\ \null $\qquad \qquad \qquad \qquad + \log(x + x^{\lambda}) \difrac{\exp[c(x + x^{\lambda})^{1 - \lambda}]}{\exp(cx^{1 - \lambda})} \difrac{x^{\lambda}}{(x + x^{\lambda})^{\lambda}} \difrac{\lambda}{x^{1 - \lambda}} +$ \\ \null \\ \null $\qquad \qquad \qquad \qquad - \log(x + x^{\lambda} - v) \difrac{\exp[c(x + x^{\lambda} - v)^{1 - \lambda}]}{\exp(cx^{1 - \lambda})} \difrac{x^{\lambda}}{(x + x^{\lambda} - v)^{\lambda}} \difrac{\lambda}{x^{1 - \lambda}} +$ \\ \null \\ \null $\qquad \qquad \qquad \qquad + \log(x + x^{\lambda} - v) \difrac{\exp[c(x + x^{\lambda} - v)^{1 - \lambda}]}{\exp(cx^{1 - \lambda})} \difrac{x^{\lambda}}{(x + x^{\lambda} - v)^{\lambda}} \difrac{\alpha x^{\lambda - 1}[\lambda \log(x) - 1]}{\log(x)^2}$ \\ \null \\ \null $\qquad \qquad \xrightarrow[x \to \infty]{} \alpha c(1 - \lambda)\exp[c(1 - \lambda)].$ \end{proof}

\begin{lemma} \label{lemmah1}

Let $\varepsilon > 0$ and define $h_1(x) = (x + \varepsilon x^{\lambda})^{1 - \lambda} - x^{1 - \lambda}$. Then

\noindent (i) $\Dom(h_1) = [0, \infty)$,

\noindent (ii) $h_1'(x) > 0$, for all $x > 0$,

\noindent (iii) $h_1(x) < \varepsilon(1 - \lambda)$, for all $x \geq 0$.

\end{lemma}

\begin{proof} Deriving $h_1$, we have $$h_1'(x) = \difrac{(1 - \lambda)[1+\varepsilon \lambda x^{\lambda - 1} - (1 + \varepsilon x^{\lambda - 1})^{\lambda}]}{(x + \varepsilon x^{\lambda})^{\lambda}},$$

\noindent which is positive for all $x > 0$ by Bernoulli's inequality. The last statement follows from lemma \ref{lemma1}. \end{proof}

\begin{lemma} \label{lemmat1}

If $a \in \R$ such that $0 < a < 1$ and $b \in \R$ such that $0 < b < 1$, then $$(1 + a) \difrac{1 - \exp(-ab)}{b} - a > \difrac{a^2}{7}.$$

\end{lemma}

\begin{proof} Expanding the left side of the inequality around $a = 0$, we have $$a^2\left(1 - \difrac{b}{2}\right) + \difrac{a^3}{6}(b-3)b - \difrac{a^4}{24}(b-4)b^2 + \difrac{a^5}{120}(b-5)b^3 - \difrac{a^6}{720}(b-6)b^4 + \cdots > \difrac{a^2}{7}.$$

\noindent Dividing both sides of this inequality by $a^2$ and rearranging the terms, this is equivalent to $$\left[\difrac{6}{7} - \difrac{b}{2} + \difrac{ab(b - 3)}{6}\right] + \difrac{a^2 b^2}{24}\left[4 - b + \difrac{ab(b - 5)}{5}\right] + \difrac{a^4 b^4}{720}\left[6 - b + \difrac{ab(b - 7)}{7}\right] + \cdots > 0.$$

\noindent Now since each of the terms inside the brackets is positive, the lemma is proved. \end{proof}

\noindent \textbf{Observation:} The author would like to thank professor Hugo Luiz Mariano for the elegant idea of the proof of the previous lemma.

\begin{lemma} \label{lemmah2}

Let $\varepsilon > 0$ and define $h_2(x) = x^{1 - \lambda} - (x - \varepsilon x^{\lambda})^{1 - \lambda}$. Then

\noindent (i) $\Dom(h_2) = \{0\} \cup [\varepsilon^{1/(1 - \lambda)}, \infty)$,

\noindent (ii) $h_2'(x) < 0$, for all $x > \varepsilon^{1/(1 - \lambda)}$,

\noindent (iii) $h_2(x) > \varepsilon(1 - \lambda)$, for all $x \geq \varepsilon^{1/(1 - \lambda)}$.

\end{lemma}

\begin{proof} Deriving $h_2$, we have $$h_2'(x) = \difrac{(1 - \lambda)[(1 - \varepsilon x^{\lambda - 1})^{\lambda} - 1 + \varepsilon \lambda x^{\lambda - 1}]}{(x - \varepsilon x^{\lambda})^{\lambda}},$$

\noindent which is negative for all $x > \varepsilon^{1/(1 - \lambda)}$ by Bernoulli's inequality. The last statement follows from lemma \ref{lemma1}. \end{proof}

\begin{lemma} \label{lemmat2}

If $a \in \R$ such that $0 < a < 1$ and $b \in \R$ such that $0 < b < 1$, then $$(1 - a) \difrac{\exp(ab) - 1}{b} - a < -\difrac{a^2}{2}.$$

\end{lemma}

\begin{proof} Analogous to that of lemma \ref{lemmat1}. \end{proof}

\section{Final thoughts} \label{section4}

A more attentive reader must have realized that the proof discovered by Newman is, in fact, an algorithm for solving a problem that can be put in a very general way: it starts with a lenght function $$\ell: [2, \infty) \rightarrow (0, \infty)$$

\noindent for which we want to know, at least asymptotically, how many prime numbers there are in the interval $$(x, x + \ell (x)].$$

The first step of the algorithm is to find a weight function $$w: \PP \rightarrow (0, \infty)$$

\noindent with none or one (or, who knows, more than one if necessary) parameter $c$ such that $w = w_c$ and it will give a weight $w(p)$ to each prime $p$. This function $w$ will generate a cumulative sum function $$W(x) = \disum_{p \leq x} w(p),$$

\noindent on which, according to \cite{RS}, we will try to show that it satisfies the expected approximation $$W(x) \sim \diint_{2}^{x} \difrac{w(t)}{\log(t)} dt,$$

\noindent because showing this relation could imply $$\pi(x + \ell(x)) - \pi(x) \sim \difrac{\ell(x)}{\log(x)}.$$

The rest of the algorithm should be easy to understand, but not necessarily to execute (in particular, on the part of the analytical extension), for those who got here. As a first test to see if this algorithm is indeed reasonably effective, the enthusiastic reader is invited to think about the following three problems.

\begin{problem}

Fix $C > 0$. What happens in the case $\ell(x) = Cx^\lambda$?

\end{problem}

\begin{problem}

Can the ideas presented here be used to prove what would be the equivalent of the PNT for arithmetic progressions in short intervals?

\end{problem}

\begin{problem}

In the light of \cite{Knopfmacher}, is it possible to reproduce what was done here in an abstract way for arithmetical semigroups and arithmetical formations?

\end{problem}

A natural continuation of this work could be done by investigating the validity of formula \ref{eq1.2} for functions $\ell(x)$ that grow slower than $x^\lambda$ for all $0 < \lambda < 1$, such as $$\exp[(\log(x))^\mu],$$

\noindent for some $0 < \mu < 1$, or even $$\log(x)^{\log(\log(x))}.$$

It is worth mentioning that, in the light of Maier's Theorem (see \cite{Maier}), this algorithm must have some limitation.

\begin{problem}

What happens in the critical case $\ell(x) = \log(x)^2$?

\end{problem}

Now, let's make three comments about result \ref{mainresult}. Firstly, an old problem from Erd\H{o}s asks (see \cite{Erdos2}) if the sequence $$\left(\difrac{p_{n + 1} - p_n}{\log(p_n)}\right)_{n \in \N}$$ of normalized differences between consecutive primes is dense in the interval $(0, \infty)$. For the author, the answer to this question may be beyond the scope of formulas \ref{eq1.1} and \ref{eq1.2}. In fact, consider the sequence $$x_n = n\log(n), \qquad n \in \N.$$

It is a nice exercise to show that this sequence satisfies formula \ref{eq1.1} for all $0 < \lambda < 1$ and for all $c > 0$ (which implies that it satisfies formula \ref{eq1.2} too), even though $$\dilim_{n \to \infty} \difrac{x_{n + 1} - x_n}{\log(x_n)} = 1.$$

Secondly, consider the following old and well-known conjecture of Legendre:

\begin{conjecture}[Legendre's conjecture] \label{legendreconjecture}

Let $n$ be a positive integer. Then there always exists at least two prime numbers between $n^2$ and $(n + 1)^2$. 

\end{conjecture}

\begin{proposition}

Legendre's conjecture is true for all $n$ sufficiently large.

\end{proposition}

\begin{proof} Since theorem \ref{mainresult} is true for $\lambda = 1/2$, then $$\pi(x + \sqrt{x}) - \pi(x) \sim \difrac{\sqrt{x}}{\log(x)}.$$ In particular, $$\dilim_{n \to \infty} \pi(n + \sqrt{n}) - \pi(n) = \infty,$$ which proves the thesis. \end{proof}

Is Legendre's conjecture true for every $n \geq 1$?

Thirdly, a beautiful problem that I've seen for the first time in \cite{Ribenboim} is the following conjecture:

\begin{conjecture}[Sierpi\'nski's conjecture] \label{matrixconjecture}

Let $n$ be an integer greater than $1$. If you write the numbers $1, 2, \ldots, n^2$ in a matrix in the following way $$\begin{matrix}
1 & 2 & \cdots & n\\
n + 1 & n + 2 & \cdots & 2n\\
2n + 1 & 2n + 2 & \cdots & 3n\\
\cdots & \cdots & \cdots & \cdots\\
(n - 1)n + 1 & (n - 1)n + 2 & \cdots & n^2
\end{matrix}$$ then each row contains, at least, a prime number.

\end{conjecture}

\begin{proposition}

Sierpi\'nski's conjecture is true for all $n$ sufficiently large.

\end{proposition}

Obviously, the proof is the same as the previous one. Is Sierpi\'nski's conjecture true for every $n \geq 2$?

Andrica's conjecture, Brocard's conjecture (that there are at least four prime numbers between the squares of two consecutives primes), and Oppermann's conjecture are also all consequences of theorem \ref{mainresult}, for all \textit{sufficiently large} integer $n$. Are all these three conjectures true for every $n \geq 2$?

Finally, allow me to use the remaining space to briefly tell where the idea for solving this problem came from.

I believe it all started when I first saw Erd\H{o}s' proof of Bertrand-Chebyshev's theorem \cite{Erdos1}. His proof is a masterpiece of mathematics, and exemplifies very clearly the idea of using a weight function to testify the existence of prime numbers in a given interval of the real line (as long as it is large enough). That's when I started looking for a suitable weight function that would effectively testify the existence of at least one prime number in the interval $$I_{\lambda, x} = (x, x + x^\lambda].$$

Initially, I tried to reproduce Erd\H{o}s' idea for the binomial coefficient $$\binom{\pin{x + x^\lambda}}{\pin{x}},$$ but then I realized that the weight function associated with this case, i.e., the product of the eventual prime numbers that lie in the interval $I_{\lambda, x}$, does not grow fast enough to guarantee the existence of at least one prime number in the interval $I_{\lambda, x}$.

So I turned my attention to Newman’s strategy to prove PNT and then it came to my mind the idea of trying a weight function with polynomial growth, i.e., studying a function of the type $$W(x) = \disum_{p \leq x} p^\omega,$$

\noindent for some $\omega > 0$, and it was at that moment that I found out the article \cite{SZ}. But even this function was not able to detect even a prime number in the interval $I_{\lambda, x}$; that's a consequence of the following theorem.

\begin{theorem} [\v{S}al\'at-Zn\'am]

If $\omega > 0$, then $$\disum_{p \leq x} p^\omega \sim \difrac{x^{1 + \omega}}{(1 + \omega)\log(x)}.$$

\end{theorem}

\begin{proof} See \cite{SZ}. \end{proof}

For me, then, the natural step to follow was to try a weight function with exponential growth, i.e., to study the function $$W(x) = \disum_{p \leq x} \exp(p).$$

But where does this function depend on $\lambda$? Furthermore, does it not grow too fast? So I asked myself: which function grows faster than any polynomial, and slower than the exponential? A natural answer was, for me, the weight function $$W(x) = \disum_{p \leq x} \exp(p^{1 - \lambda}).$$

With this function it is already possible to prove (following the same procedure as in Step 1) that $$\difrac{1-\exp(\lambda - 1)}{1 - \lambda} \leq \diliminf_{x \to \infty} \difrac{\pi_A(x + x^\lambda) - \pi_A(x)}{x^\lambda/\log(x)} \leq \dilimsup_{x \to \infty} \difrac{\pi_A(x + x^\lambda) - \pi_A(x)}{x^\lambda/\log(x)} \leq \difrac{\exp(1 - \lambda) - 1}{1 - \lambda},$$

\noindent which indicates, in a certain way, the need to insert the parameter $c$ in the function $w$. This brings us to the function $$W(x) = \disum_{p \leq x} \exp(cp^{1 - \lambda}).$$

This function is already good enough to make all the steps described in section \ref{section2} work (including steps $8$ and $9$ of Newman's proof), except for step $10$, in which its respective function $\Psi$ starts to be written using other functions, which complicates things a lot. The solution to this technicality was then to try to find a function $w$ whose associated cumulative function had an asymptotic behavior similar to the function $\exp(cx^{1 - \lambda})$; that’s the case of $$W(x) = \disum_{p \leq x} \difrac{c(1 - \lambda) \log(p) \exp(cp^{1 - \lambda})}{p^{\lambda}}.$$

I was amazed at how well this function behaves in relation to the necessary adaptations proposed here (in particular, by step 3, whose change of variables is explicit), but then I got stuck looking for how to realize step 8. After many unsuccessful attempts, I thought about expanding $\Psi(s)$ around $\lambda = 1$, as follows: first, write $$\Psi(s) = \disum_p \difrac{\log(p)}{p^\lambda \exp\left[s\left(\difrac{p^{1 - \lambda} - 1}{1 - \lambda}\right)\right]} = \disum_p \difrac{\log(p)}{p^{1+s}} \cdot \difrac{p^{s + 1 - \lambda}}{\exp\left[s\left(\difrac{p^{1 - \lambda} - 1}{1 - \lambda}\right)\right]}$$

Now, expanding the rightmost term around $\lambda = 1$ for each prime $p$, one gets \begin{equation} \label{eq4.1} \Psi(s) = \disum_p \difrac{\log(p)}{p^{1 + s}} \cdot \disum_{j = 0}^{\infty} \difrac{(\lambda - 1)^j \log(p)^j f_j(s\log(p))}{j!},\end{equation} in which these $\{f_j(x)\}_{j \geq 0}$ are a sequence of polynomials in $\mathbb{Q}[x]$ defined recursively by $$f_0(x) = 1 \qquad \qquad \text{and} \qquad \qquad f_{j+1}(x) = -f_j(x) + x\disum_{k = 0}^{j} \dibinom{j}{k} \difrac{(-1)^k \cdot f_{j-k}(x)}{k+2},$$ for all $j \geq 0$, additionally having the marvelous property that the function $$s \mapsto \disum_p \difrac{\log(p)^{j + 1} f_j(s\log(p))}{p^{s + 1}},$$ initially defined for $\Re(s) > 0$, extends analytically to $\Re(s) \geq 0$, for all $j \geq 1$. This makes us very willing to change the order of the sums in equation \ref{eq4.1}. Indeed, if we could do that, we would have \begin{equation*} \begin{split} \Psi(s) & = \disum_{j = 0}^{\infty} \difrac{(\lambda - 1)^j}{j!} \disum_p \difrac{\log(p)^{j + 1} f_j(s\log(p))}{p^{1 + s}} \\ & = \Phi(1+s) + \disum_{j = 1}^{\infty} \difrac{(\lambda - 1)^j}{j!} \disum_p \difrac{\log(p)^{j + 1} f_j(s\log(p))}{p^{1 + s}}, \end{split} \end{equation*} in such a way that by subtracting $\difrac{1}{s}$ from each side of this equation, we would have $$\Psi(s) - \difrac{1}{s} = \Omega(s) + \disum_{j = 1}^{\infty} \difrac{(\lambda - 1)^j}{j!} \disum_p \difrac{\log(p)^{j + 1} f_j(s\log(p))}{p^{1 + s}},$$ thereby proving theorem \ref{mainresult}.

Even after many months of thinking about it, I was unable to prove that it was possible to interchange the order of the sums. It was then that I had the idea of using the Mellin transform, but I was quickly warned that using the second primitive of the function $\tau(s)$ might not guarantee the non-existence of a pole or a branch point at $s = 0$. I was quite saddened by this and returned to the double sum. For a moment, I glimpsed the possibility of actually exchanging the order of the sums (which led me to write the third version of the article), but unfortunately this was an oversight on my part. After a few more months of reflection, the possibility of taking the third antiderivative of the function $\tau(s)$ (as a process of regularization of singularities through an integral transform) arose, which I believe finally resolves this journey.

The rest of the story is in the manuscript in hand. We close the presentation of this article by leaving the following question to the reader.

\begin{problem}

How does the function $\tau(s)$ behave in a neighborhood of $\Re(s) = 0$ in the limit case $\lambda = 0$, i.e., how does the function $$s \mapsto \disum_p \difrac{\log(p)}{\exp(sp)}$$ behave near $\Re(s) = 0$? Does this have any meaning in relation to the distribution of prime numbers?

\end{problem}

\subsection*{Acknowledgments}

The author would like to thank his colleagues at IFSP, Itaquaquecetuba campus, for the opportunity to undertake his postdoctoral research at IME - USP, and his supervisor, Hugo Luiz Mariano, for his unwavering support during this period. He is grateful to Professor Kai (Steve) Fan for identifying an inconsistency in the first draft of this article (specifically in the proof of proposition \ref{step2}), and to Professor Carl Pomerance for his insightful attention to this matter. Special thanks are also due to Professor Gregory Debruyne for pointing out a flaw regarding the use of the inverse Mellin transform in the second version of this article, and to Professors Gregory Debruyne (again), Jean-Christophe Pain, Sinai Robins and Christian T\'afula Santos for noting that the locally uniform convergence of the double sum in the third version might not sufficiently warrant interchanging the order of summation. Finally, the author expresses his gratitude to Professor Barry Mazur for the time and effort spent reviewing this manuscript.

\end{document}